\documentclass[a4paper,12pt]{amsart}
\usepackage{amsmath,amssymb,amsthm}
\usepackage[left=2cm,right=2cm]{geometry}
\makeatletter
\@namedef{subjclassname@2020}{%
\textup{2020} Mathematics Subject Classification}
\makeatother
\newtheorem{thm}{Theorem}
\newtheorem{prop}[thm]{Proposition}
\newtheorem{lemma}[thm]{Lemma}
\newtheorem{cor}[thm]{Corollary}

\def\w{\omega_{\Omega}}
\def\C{\mathbb C}
\def\Cn{\mathbb{C}^n}
\def\el{l_e}
\def\kl{l^\Omega_k}
\def\krl{l^\Omega_{\kappa}}

\def\kr{\kappa_\Omega}

\def\O{\Omega}

\def\endofproof{\hfill \square}

\def\g{g_\Omega}

\def\sdd{\sqrt{\delta_{\Omega}(x)\delta_{\Omega}(y)}}

\theoremstyle{remark}
\newtheorem{rem}{Remark}
\theoremstyle{definition}
\newtheorem{defn}{Definition}

\begin{document}

\title[Strongly Goldilocks domains and quantitative visibility]
{Strongly Goldilocks domains, quantitative visibility, and applications}

\author{Nikolai Nikolov and Ahmed Yekta Ökten}

\address{N. Nikolov\\
Institute of Mathematics and Informatics\\
Bulgarian Academy of Sciences\\
Acad. G. Bonchev Str., Block 8\\
1113 Sofia, Bulgaria}
\email{nik@math.bas.bg}
\address{Faculty of Information Sciences\\
State University of Library Studies
and Information Technologies\\
69A, Shipchenski prohod Str.\\
1574 Sofia, Bulgaria}

\address{A.Y.Ökten\\
Institut de Math\'ematiques de Toulouse; UMR5219 \\
Universit\'e de Toulouse; CNRS \\
UPS, F-31062 Toulouse Cedex 9, France} \email{ahmed$\_$yekta.okten@math.univ-toulouse.fr}

\begin{abstract} In the last decade, G. Bharali and A. Zimmer defined a class of domains called Goldilocks domains and they showed that such a domain satisfies a visibility condition with respect to the Kobayashi extremal curves. Inspired by their construction, we define a subclass of Goldilocks domains called strongly Goldilocks domains and we prove a quantitative visibility result on strongly Goldilocks domains. Using our quantitative visibility result, we extend the Gehring-Hayman theorem on simply connected planar domains to strongly Goldilocks domains. As an application of our construction, we also give lower estimates to the Kobayashi distance.
\end{abstract}

\thanks{The first named author is partially supported by the National Science Fund,
	Bulgaria under contract KP-06-N52/3. This work was started during the stay of the
	second name author at the Institute of Mathematics and Informatics,
	Bulgarian Academy of Sciences (March 2022) in the framework of the same contract. \\
The second author received support from the University Research School EUR-MINT (State support managed
by the National Research Agency for Future Investments program bearing the reference ANR-18-EURE-0023).}



\maketitle
\section{Introduction}
It is well-known that the unit ball endowed with the Kobayashi distance is negatively curved and strongly pseudoconvex domains resemble the unit ball close to the boundary. Therefore it is reasonable to expect the Kobayashi geometry of strongly pseudoconvex domains to exhibit a similar behaviour. However, it is often not possible to study local geometry of the Kobayashi distance. Balogh and Bonk \cite{BB} showed that the Kobayashi geometry of strongly pseudoconvex domains are hyperbolic in a global sense, that is, strongly pseudoconvex domains are Gromov hyperbolic w.r.t. the Kobayashi distance. Zimmer\cite{Z} and Fiacchi\cite{FI} extended their result to smooth bounded convex domains of finite type and smooth bounded pseudoconvex domains of finite type in $\mathbb{C}^2$.

Another hyperbolic behaviour of geodesic metric spaces is the visibility property. Informally, it means that geodesics joining points close to different boundary points tend to bend inside the metric space. The paper \cite{BZ} defined the Goldilocks domains by demanding growth conditions on the Kobayashi distance and the Kobayashi-Royden pseudometric with no assumption of boundary regularity. There, they showed that Goldilocks domains endowed with the Kobayashi distance satisfy a variant of the visibility property w.r.t. their Euclidean boundaries and provided some applications of visibility. See also \cite{BM,BG,BGZ,BNT} and references therein.

In this paper, we define a subclass of Goldilocks domains called strongly Goldilocks domains (see Definition \ref{stronggoldi}). Roughly speaking a strongly Goldilocks domain $\O$ is a bounded domain with a sufficiently regular boundary and its Kobayashi-Royden pseudometric grows faster than a function $1/\w,$ where $\w$ satisfies certain growth conditions. In particular, the class of strongly Goldilocks domains contains strongly pseudoconvex domains, smooth bounded pseudoconvex domains of finite type and $m$-convex domains with Dini-smooth boundary.

We first provide a quantitative visibility result on a strongly Goldilocks domain $\O.$ Explicitly, we prove Theorem \ref{firstbig} which, , shows that if $\gamma:I\longrightarrow \Omega$ is a $\lambda$-geodesic (see Definition \ref{lambdageodesics}) joining $x,y\in\Omega,$ then $\max\delta_\Omega\circ\gamma\geq c \g^{-1}(\|x-y\|),$ where $\g$ is a function associated to $\Omega$ derived from $\w$ and the constant $c>0$ depends only on $\O$ and $\lambda$. Then Proposition \ref{cconvex} gives an explicit form of this estimate on strongly pseudoconvex domains, smooth bounded pseudoconvex domains of finite type and $m$-convex domains with Dini-smooth boundary.

With our quantitative visibility condition, we give an extension of the classical Gehring-Hayman inequality to convexifiable strongly Goldilocks domains. Namely, our Theorem \ref{secondbigcfb} shows that if $\Omega$ is a convexifiable strongly Goldilocks domain with $\gamma$ and $\g$ as above, we have  $\el(\gamma)\leq  C\g(\|x-y\|),$ where $\el(\gamma)$ denotes the Euclidean length of the curve $\gamma$ and $C>0$ depends only on $\O$ (see Corollary \ref{betterboundsthanchineseguys} for some explicit estimates). This shows that on convexifiable strongly Goldilocks domains, Kobayashi geodesics are also short in Euclidean length. As the situation for convexifiable domains is more intricate than for convex domains, we study the problem in two separate subsection.

Finally, as another application of our quantitative visibility condition, we prove Theorem \ref{lowerboundfinal} which provides a relatively sharp lower bound to the Kobayashi distance on strongly Goldilocks domains.

\section{Preliminaries}\label{preelim}
\subsection{Notation} $\:$

	(i) Let $\Omega$ be a domain in $\Cn$, then for $x\in\Omega$, $v\in\Cn$ we set $$\delta_\Omega(x):=\displaystyle{\inf\{\lVert x-y \rVert: y\in\partial\Omega\}} \:\:\:\:\: \text{and} \:\:\:\:\: \delta_\Omega(x;v):=\displaystyle{\inf\{| \lambda | : x+\lambda {v} \in \partial \Omega, \: \lambda\in \mathbb{C}\}}. $$
	
	(ii) For a curve $\gamma:[0,1]\to\O$ we set $D_\gamma:=\max_{t\in[0,1]}\delta_\O(\gamma(t)).$
	
	(iii) Let $f,g$ be real valued functions taking non-negative values on a set $X$. We say $f \lesssim g$ or $g \gtrsim f$ if there exists $C>0$ such that $f(x) \leq C g(x)$ for all $x\in X$. We say $f\sim g$ if $f\lesssim g$ and $f \gtrsim g$.
\subsection{The Kobayashi distance} Let $\Omega$ be a domain in $\Cn$. Recall that the Lempert function of $\Omega$ is defined as $$l_\Omega(x,y):=\inf\{\tanh^{-1}(|\alpha|): f\in\mathcal{O}(\Delta,\Omega), \:f(0)=x,\:f(\alpha)=y\},$$ where $\mathcal{O}(\Delta,\Omega)$ denotes the set of holomorphic functions on the unit disc $\Delta$ into $\Omega$.
	
In general the Lempert function does not satisfy the triangle inequality. The Kobayashi pseudodistance is the largest pseudodistance not exceeding the Lempert function. The Kobayashi length of a continuous curve $\gamma:[0,1]\rightarrow\Omega$ is defined as $$ \kl(\gamma):=\displaystyle{\inf_{n\in\mathbb{N}}\left\{\sum_{j=1}^n l_\Omega(x_{j-1},x_j):x_0=\gamma(0), x_n=\gamma(1), \: x_j\in\gamma([0,1])\right\}} $$ and the Kobayashi pseudodistance on $\Omega$ is given as
$$k_\Omega(x,y):=\displaystyle{ \inf \kl(\gamma) },$$ where the infimum is taken over all continuous curves joining $x,y\in\Omega$.

Another relevant notion is the Kobayashi-Royden pseudometric. For $x\in \Omega$ and $v\in\Cn$ it is defined as
$$\kappa_\Omega(x;v)=\inf\{|\alpha|:\exists\varphi\in\mathcal O(\Delta,\Omega):
\varphi(0)=x,\alpha\varphi'(0)=v\}.$$ Due to \cite[Theorem 1]{R} we have
$$k_\Omega(x,y)=\displaystyle{\inf \{ \krl(\gamma): \gamma \: \text{is piecewise}\:\mathcal{C}^1,\: \gamma(0)=x,
\:\gamma(1)=y\}},
$$
where $$\krl(\gamma):=\int_0^1\kappa_\Omega(\gamma(t);\gamma'(t))dt.$$

The Kobayashi-Buseman pseudometric (see e.g. \cite[p. 135]{JP}) $\hat k_\O$ is the largest pseudometric with convex indicatrix not exceeding $k_\O$. In particular (see e.g. \cite[proof of Theorem 3.6.4]{JP}), if $\gamma:[0,1]\to\O$ is a piecewise $\mathcal{C}^1$-curve, then
\begin{equation}\label{eqn:comparinglengths}
	 l^\O_k(\gamma)\leq l^\O_{\hat\kappa}(\gamma)\leq l^\O_\kappa(\gamma),
\end{equation}
where $l^\O_{\hat{\kappa}}(\gamma):=\int_0^1 \hat \kappa_\O(\gamma(t);\gamma'(t))dt$.

A domain $\Omega\subset\Cn$ is said to be hyperbolic if $k_\Omega$ is a distance and complete hyperbolic if $(\Omega,k_\Omega)$ is complete as a metric space. The Hopf-Rinow theorem \cite[Proposition 3.7]{BH} asserts that if $\Omega$ is complete hyperbolic, then $(\Omega,k_\O)$ is a geodesic space, that is for any $x,y\in\O$ there exists a continuous curve $\gamma:[0,1]\to\O$ joining $x$ and $y$ such that $k_\Omega(x,y)=l^k_\O(\gamma)$. We call such a curve a Kobayashi geodesic. As in general Kobayashi geodesics need not exist, the authors in \cite{BZ} defined $(\lambda,\epsilon)$-quasi-geodesics and $(\lambda,\epsilon)$-almost-geodesics.
In this paper, we use the concept of $\lambda$-geodesics, whose definition we give below.

\begin{defn}\label{lambdageodesics}
	Let $\Omega\subset\Cn$ be a hyperbolic domain, $\lambda \geq 1$. We say that a curve $\gamma:[0,1]\to\Omega$ is a \textit{$\lambda$-geodesic} for $\kappa_\O,$ resp. $\hat\kappa_\O,$ if for all $s,t\in [0,1]$ we have
	 $$ \krl(\gamma|_{[s,t]}) )\leq \lambda k_\Omega(\gamma(s),\gamma(t)),\quad\text{resp. }l^\O_{\hat{\kappa}}(\gamma|_{[s,t]}) \leq \lambda k_\Omega(\gamma(s),\gamma(t)).$$
\end{defn}

It is easy to see that any $\lambda$-geodesic is Lipschitz, so by taking a reparametrization, one sees that $\lambda$-geodesics for $\kappa_\O$ correspond to $(\lambda,0)$-almost-geodesics in the sense of \cite{BZ}. Moreover, by \eqref{eqn:comparinglengths}, any $1$-geodesic is a Kobayashi geodesic.

\begin{rem}\label{rem:1geodesics} By \cite[Proposition 4.6]{BZ}, if $\O$ is a hyperbolic domain and $\gamma:[0,1]\to\O$ is a Kobayashi geodesic, then $\gamma$ is Lipschitz and $l^\O_k(\gamma) = l^\O_{\hat\kappa}(\gamma)$. Thus, in the case of hyperbolicity, any Kobayashi geodesic is
a $1$-geodesic for $\hat\kappa_\O$ and vice versa. It remains unclear whether the same is true for $\kappa_\O.$
\end{rem}

From now on, $\lambda$-geodesic will mean $\lambda$-geodesic for $\kappa_\O.$  However the reader can verify that our results also remain true for $\lambda$-geodesics for $\hat\kappa_\O.$

\begin{defn}	\label{growth}
		Let $\O$ be a domain in $\mathbb{C}^n$. We say that $(\Omega,k_\Omega)$ has \textit{$\alpha$-growth} if there exist  some  $\alpha>0$
		and $x_0\in \O$ such that
		$$
		\sup_{x\in \Omega} \left( k_\Omega (x_0,x)- \alpha\log \frac{1}{\delta_\Omega(x)} \right) <\infty.
		$$
\end{defn}
Let $p\in\partial \Omega$ be a $\mathcal{C}^1$-smooth boundary point. We say that $p$ is a Dini-smooth boundary point if the inner unit normal vector to $\partial \Omega$ is a Dini-continuous function near $p$. This means that there exists a neighbourhood $U$ of $p$ such that $\int_{0}^{\epsilon_0} \frac{ w(t)}{t}dt<\infty$, where $$w(t):=\sup\{\|n_p-n_q\|:p,q\in \partial\O\cap U, \:\:\: \|p-q\| < t\}$$ and $n_p,n_q$ denotes the inner unit normal to $\partial\O$ taken at $p,q$ respectively. We say that $\Omega$ has Dini-smooth boundary if any $p\in\partial \Omega$ is a Dini-smooth boundary point. It is clear that any $\mathcal{C}^{1,\alpha}$-smooth boundary point is a Dini-smooth boundary point. The following estimate is crucial for our purposes.

\begin{lemma}\label{dinilemma}\cite[Corollary 8]{NA}
	Let $\Omega\subset\Cn$ be a domain with Dini-smooth boundary. Then there exists $C>0$ such that we have $$k_\Omega(x,y) \leq \log\left(1+\dfrac{C\|x-y\|}{\sqrt{\delta_\Omega(x)\delta_\Omega(y)}}\right), \:\:\:\:\: x,y \in \O .$$
\end{lemma}
Notice that the lemma above implies that if $\O\subset\Cn$ is a domain with Dini-smooth boundary, then $(\O,k_\O)$ has $1/2$-growth.

The next lemma follows from the proof of \cite[Theorem 5.4]{B} (see also \cite[(2)]{N2}).
\begin{lemma}\label{convexlemma}
	Let $\Omega$ be a bounded convex domain. We have $$ k_\Omega(x,y)\geq \dfrac{1}{2}\left|\log\left( \dfrac{\delta_\Omega(x)}{\delta_\Omega(y)} \right)\right|, \:\:\:\:\: x,y \in \O .$$
\end{lemma}

The upper bound below easily follows for any domain; for a proof of the lower bound see e.g. \cite[Lemma 2.5]{Z}.

\begin{lemma}\label{convexlemmafund} Let $\Omega\subset\Cn$ be a convex domain. We have
	$$ \dfrac{1}{2\delta_\Omega(x;v)}\leq \kappa_\Omega(x;v) \leq \dfrac{1}{\delta_\Omega(x;v)}, \:\:\:\:\: x\in \O, \:\:\: v \in \Cn . $$
\end{lemma}

The following estimate is mentioned in \cite[p.~633]{NTR} and \cite[Remark 2.7]{Z}. For convenience of the reader,
we include a proof.

\begin{lemma}\label{lem:proofisgiven} Let $\Omega$ be a convex domain in $\Cn$. We have $$ k_\Omega(x,y) \geq \dfrac{1}{2}\log\left(1+\dfrac{1}{\min\{\delta_\Omega(y;{x-y}),\delta_\Omega(x;{x-y})
\}}\right), \:\:\:\:\: x,y \in \O.$$
\end{lemma}

\begin{proof}
	Without loss of generality assume that $\delta_\Omega(x;x-y)\geq\delta_\Omega(y;x-y)$. Let $L$ denote the complex line passing through $x$ and $y$, and $p\in L\cap\partial\Omega$ be with $\|y-p\|=\|x-y\|\delta_\Omega(y;{x-y})$. Set $\pi:\Cn\rightarrow L$ to be the projection of $\Cn$ onto $L$ in the direction of a supporting hyperplane to $\partial\Omega$ at $p$. Note that such an hyperplane exists due to convexity of $\Omega$.
	
As $\pi(\O)$ is convex and $\delta_{\pi(\Omega)}(y)=\|y-p\|$, it follows by \cite[Proposition 2]{NTR} that
	$$ k_\Omega(x,y)\geq k_{\pi(\Omega)}(x,y)\geq \dfrac{1}{2}\log\left(1+\frac{1}{\delta_\Omega(y;x-y)}\right) .$$

\end{proof}
\subsection{Gromov hyperbolicity and the visibility property}

\begin{defn} Let $(X,d)$ be a metric space. The Gromov product of $x,y\in X$ w.r.t. $o\in X$ is defined to be $$ (x|y)_o:=\dfrac{1}{2}(d(x,o)+d(o,y)-d(x,y)).$$
	A metric space $(X,d)$ is said to be \textit{Gromov hyperbolic} if for some $\delta>0$ and any $ x,y,z,o\in X$ we have
	$$(x|y)_o \geq \min\{(x|z)_o,(z|y)_o\}-\delta.$$
\end{defn}

We are interested in the case $X=\O$ and $d=k_\O$. We say that a domain $\O\subset\Cn$ is Gromov hyperbolic if $(\O,k_\O)$ is a Gromov hyperbolic
space. Note that the strongly pseudoconvex domains and the smooth bounded pseudoconvex domains of finite type in $\mathbb{C}^2$ are Gromov hyperbolic by \cite[Theorem 1.4]{BB} and \cite[Theorem 1.1]{FI} respectively. Furthermore, \cite[Theorem 1.1]{Z} shows that a smooth bounded convex domain is Gromov hyperbolic if and only if it is of finite type.

The following follows from the more general geodesic stability lemma (see e.g. \cite[p. 401, Theorem 1.7]{BH}).

\begin{lemma}\cite[Theorem 1.7]{BH}\label{geodesicstabilitylemma}
	Let $\O$ be a complete hyperbolic domain which is Gromov hyperbolic. Let $\gamma:[0,1]\rightarrow \O$ be a Kobayashi geodesic and let $\sigma:[0,1]\to\O$ be a $\lambda$-geodesic joining the same two points. Then there exists $K>0,$ depending only on $\O$ and $\lambda,$ such that for all $t\in[0,1]$ there exists $s\in[0,1]$ such that $k_\O(\gamma(t),\sigma(s))\leq K$.
\end{lemma}

We refer the reader to \cite[Part III]{BH} for a detailed discussion on Gromov hyperbolic metric spaces and also to the recent expository article \cite{BG} where the authors discuss constructions related to Gromov hyperbolicity on complex domains.

Another relevant notion is the visibility property, we follow the definition given in \cite{BNT}.
\begin{defn}
	\label{bntvisibility}
	Let $\Omega$ be a complete hyperbolic domain in $\Cn$. Let $p,q \in \partial \Omega$, where $p\neq q$. We say that the pair $\{p,q\}$ has \textit{visible
	geodesics}
	if there exist neighborhoods $U, V$ of $p, q$ respectively such that $\overline{U} \cap \overline{V} = \emptyset$,
	and a compact set $K \subset \Omega$ such that for any geodesic $\gamma:[0,1]\rightarrow \Omega$
	with $\gamma(0)\in U$, $\gamma(1)\in V$, then $\gamma ([0,1]) \cap K \neq \emptyset$.
	
	We say that $\Omega$ satisfies \textit{the visibility property} if any pair $\{p,q\} \subset \partial \Omega$, where $p\neq q$
	has visible  geodesics.
\end{defn}

The definition given above only makes sense when the domain is complete hyperbolic. In \cite{BZ} (see also \cite{BM}), the authors gave a more general definition of the visibility property. There they also defined the following class of domains and showed that they satisfy the visibility property in their own context. We recall their definition below.

\begin{defn}\label{goldilocksdefn}
	Let $\O$ be a bounded domain in $\Cn$. We say that $\Omega$ is \textit{Goldilocks} if:
	
(i)		$$ \displaystyle{\int_0^{\epsilon_0}}\frac{M_\Omega(r)}{r}dr < \infty ,$$ where $$M_\Omega (r):=\sup\left\{\frac{1}{\kappa_\Omega(x;v)}: \|v\|=1, \: \delta_{\Omega}(x)\leq r  \right\};$$

(ii) $(\Omega,k_\Omega)$ has $\alpha$-growth for some $\alpha > 0$.
\end{defn}

Notice that the definition above intuitively gives the reason why the geodesics (or curves that approximate geodesics in a suitable sense) tend to spend as little time as possible near the boundary, the Kobayashi-Royden pseudometric grows too quickly but the Kobayashi distance grows too slowly for geodesics to stay close to the boundary.

Recall a geometric notion given in \cite{BZ}. Roughly speaking we say  that a domain $\Omega$ satisfies the interior cone condition if any point near its boundary belongs to the central axis of a uniform cone centered at a boundary point. \cite[Lemma 2.3]{BZ} shows that Condition 2 in Definition \ref{goldilocksdefn} is satisfied by domains that satisfy the interior cone condition. Moreover, it is not difficult to see that uniform cone condition is satisfied by domains with Lipschitz boundaries, in particular for convex domains.

On the other hand, Condition 1 is satisfied for a large class of domains including bounded pseudoconvex domains of finite type and bounded convex domains with a boundary that is not too flat in the precise sense given in \cite[Lemma 2.9]{BZ}. This discussion shows that although its definition is rather abstract, it is reasonable to expect that "regular" domains satisfy Goldilocks conditions. As a partial converse, by \cite[Proposition 2.15]{BZ} we have that the Goldilocks domains are pseudoconvex.

\section{Strongly Goldilocks domains}\label{qvisibility}
In this section we introduce and give examples of strongly Goldilocks domains, and provide a quantitive visibility condition for such domains.

\subsection{Definition and properties}

\begin{defn}\label{stronggoldi} Let $\Omega$ be a \emph{bounded} domain in $\Cn$. Set $d_\O:=\max\{\delta_\O(x):x\in\O\}$.
We say that $\Omega$ is \emph{strongly Goldilocks} (SG) if:
	
(i) $\Omega$ has Dini-smooth boundary;

(ii) The Kobayashi-Royden pseudometric of $\Omega$ satisfies
\begin{equation}\label{lb}
\kr(x;v)\geq\dfrac{\|v\|}{\omega_\O(\delta_\Omega(x))},\quad x\in\O,\ v\in\Bbb C^n,
\end{equation}
where $\w:(0,d_\O]\rightarrow\mathbb{R}^{+}$ is a strongly Goldilocks (SG) function, that is,

	(a) $\frac{\w(r)}{r}$ is decreasing and $\frac{w(d_\O)}{d_\O}\ge e;$
	
	(b) $\w(r)\log\left(\frac{\w(r)}{r}\right)$ is increasing;
	
	(c) $$\int_0^{d_\O}\frac{\w(r)}{r}\log\left(\frac{\w(r)}{r}\right)dr < \infty. $$
\end{defn}

Although the conditions above seem somewhat restrictive we will later see that they are natural in the sense that they hold for relatively large class of domains.

\begin{rem}\label{rem:remarkaboutSGdomains}$\:$ Note that (a) and (b) above imply that the function $\w$ is strictly increasing.
\end{rem}

\begin{prop} If $\Omega\subset\Bbb C^n,$ $n\ge 2,$ and \eqref{lb} holds, then $\w(r)\gtrsim\sqrt r.$
\end{prop}

\begin{proof} Let $q\in\Omega$ and $p\in\partial \Omega$ be such that $\|p-q\|=\delta_\Omega(q)=:c$.
Denote by $n_p$ the unit inner normal vector to $\partial \Omega$ at $p.$ Set $B$ to be the Euclidean ball with center
at $q$ and radius $c$. Then for any $v\in T_p^{\Bbb C}(\partial \O)$ and $r\in(0,2c)$
we have
$$\dfrac{\|v\|}{\sqrt{(2c-r)r}}=\kappa_B(p+rn_p;v)\geq \kappa_\Omega(p+rn_p;v)\geq
\dfrac{\|v\|}{\w(r)}.$$
\end{proof}

Observe also that $\w \gtrsim M_\Omega,$ where $M_\Omega$ comes from Definition \ref{goldilocksdefn}.
As SG domains have Dini-smooth boundary, \cite[Lemma 2.3]{BZ} shows that the set of SG domains form a strict subset of Goldilocks domains. In particular, by \cite[Proposition 2.15]{BZ} and \cite[Theorem 1.4]{BZ}, they are pseudoconvex and they satisfy the visibility property.

We present some observations which will be useful in what will follow.
\begin{prop}\label{propertiesofstronggoldi} Let $\omega:(0,d]\to \mathbb{R}^+$ be a SG function. For $r\in(0,d],$ set
$$g_\omega(r):=\int_{0}^{r}\frac{\omega(t)}{t}\log\left(\frac{\omega(t)}{t}\right)dt.$$
Then:

		\emph{(i)} $g_\omega$ is a strictly increasing function, so it has an inverse;
		
		\emph{(ii)} $g_\omega(r)\ge \omega(r);$
		
		\emph{(iii)} $\frac{g_\omega(r)}{r}$ is a decreasing function.
\end{prop}
\begin{proof} (i) This statement is obvious.

(ii) Since $h(t):=\frac{\omega(t)}{t}\ge e$ is decreasing, we have $$g_\omega(r)=\int_{0}^{r}\dfrac{w(t)}{t}\log\left(\frac{\omega(t)}{t}\right)dt\ge\frac{\omega(r)}{r}\int_{0}^{r}dt=\omega(r).$$

(iii) It follows from
$$\left(\frac{g_\omega(r)}{r}\right)'=\frac{\int_{0}^{r}(h(r)\log h(r)-h(t)\log h(t))dt}{r^2}\le 0.$$

\end{proof}
Note that by the properties of functions $\omega,g_\omega$ we see that if $K\geq1$
\begin{equation}\label{wecangettheconstantout}
	\omega(Kr)\leq K\omega(r) \:\:\: \text{and} \:\:\: g_\omega(Kr)\leq Kg_\omega(r).
\end{equation}
This property will be important in our estimates in the upcoming sections.

\subsection{Examples}

Recall that a convex domain $\Omega\subset\Cn$ is said to be $m$-convex if there exists $C>0$ such that
for every $x\in\Omega$ and $v\in\Cn$ with $\|v\|=1$ we have
$$\delta_\O(x;v)\le C\delta_\O^{1/m}(x).$$
This notion extends the notion of finite type in the case of smooth convex domains (see e.g. \cite{Z}).

We point out that
\begin{equation}\label{eqn:specialcases}
	\kappa_\Omega(x;v)\geq \dfrac{c\|v\|}{\delta^{\epsilon}_\Omega(x)},\quad x\in\O,\ v\in\Cn,
\end{equation}
with:

-- $\epsilon=1/2$ if $\O$ is strongly pseudoconvex (see e.g. \cite[Theorem 10.4.2]{JP});

-- $\epsilon=1/m$ if $\O$ is $m$-convex (by Lemma \ref{convexlemmafund});

-- $\epsilon=1/m$ if $\O\subset\Bbb C^2$ is pseudoconvex of finite type $m$ (see \cite[Theorem 1]{Cat});

-- $\epsilon>0$ if $\O\subset\Bbb C^n$ is pseudoconvex of finite type (see \cite[Theorem 1]{Cho}).

Note that such domains are SG (assuming Dini smoothness if $\O$ is $m$-convex) as the following
proposition shows.

\begin{prop}\label{cconvex} Let $\epsilon\in(0,1)$ and $d>0$. Then
the function $\omega(r):=Cr^\epsilon$ is SG on (0,d] if and only $C\ge d^{1-\epsilon} e^{c_\epsilon}$, where $c_\epsilon:=\max\{{-1+1/\epsilon},1\}$.
Moreover, for a fixed constant $c\in(0,1),$
$$g_\omega(r)\sim r^\epsilon\log(1/r)\quad\text{and}\quad g^{-1}_\omega(r)\sim\left(\frac{r}{\log(1/r)}\right)^{1/\epsilon},
\quad 0\le r\le c.$$
\end{prop}

\begin{proof} The first statement follows directly. Straightforward calculations lead to
$$g_\omega(r)=\frac{Cr^\epsilon}{\epsilon}\left((1-\epsilon)\log\frac{1}{r}+\frac{1-\epsilon}{\epsilon}+\log C\right).$$
Then
$$\lim_{r\to 0}\frac{g_\omega(r)}{r^\epsilon\log(1/r)}=\frac{C(1-\epsilon)}{\epsilon},\quad
\lim_{r\to 0}\frac{g^{-1}_\omega(r)}{(\frac{r}{\log(1/r)})^{1/\epsilon}}=
\lim_{s\to 0}\frac{s}{(\frac{g_\omega(s)}{\log(1/g_\omega(s))})^{1/\epsilon}}=\left(\frac{\epsilon}{C(1-\epsilon)}\right)^{1/\epsilon}$$
and the proof easily follows.
\end{proof}

So we may say that a domain $\Omega$ is $SG_\epsilon$ if it has Dini-smooth boundary and \eqref{eqn:specialcases} holds with $\epsilon\in(0,1/2].$

\subsection{Quantitative visibility} The goal of this subsection is to present Theorem \ref{firstbig}, which can be considered to be the main theorem of this paper. It concerns SG domains and it gives a way to measure how much the geodesics tend to go inside the domain, in other words, it gives a quantitative version of visibility. Notice that it roughly implies that the larger $\omega_\O$ is the "less" visible $\O$ is. In the convex case, by Lemma \ref{convexlemmafund}, this situation occurs when the boundary gets more and more flat. Moreover, it also shows that the geodesics which join points that are further away in the Euclidean distance tend to stay further away from the boundary.

From now on, for a SG domain $\Omega$ with a SG function $\omega_\Omega,$ we set $\g:=g_{\omega_\Omega}.$

	\begin{thm}\label{firstbig}
		Let $\Omega$ be a SG domain and $\gamma:[0,1]\rightarrow\Omega$ be a $\lambda$-geodesic. Then there exists $C>0,$ \textit{depending only on $\O$ and $\lambda,$}  such that $$ \el(\gamma) \leq C \g(D_\gamma) .$$
	\end{thm}

Recall that $D_\gamma=\max_{t\in[0,1]}\delta_\O(\gamma(t))$ and $\el(\gamma)$ denotes the Euclidean length of $\gamma$.

Let $\Omega$ be as above and $\gamma:[0,1]\rightarrow \Omega$ be a $\lambda$-geodesic on $\Omega$ which joins $x\in\Omega$ to $y\in\Omega$. We want to compare its Euclidean length with the quantity $D_\gamma$ in terms of the function $\g$. To do so, we will first divide $\gamma$ into pieces with endpoints of comparable boundary distance.

Set $D:=D_\gamma$ and let $N_x,N_y$ denote the integers such that $$ \dfrac{D}{e^{N_x}}\leq \delta_{\Omega}(x)<\dfrac{D}{e^{N_x-1}} \: \text{and} \: \dfrac{D}{e^{N_y}}\leq \delta_{\Omega}(y)<\dfrac{D}{e^{N_y-1}}.$$

Travelling from $x$ to $y$, for $i\in\{-N_x+1,0\}$, set  $t_i=\min\{t\in [0,1]: \delta_\Omega(\gamma(t))=D e ^{i} \}$, $x_i:=\gamma(t_i)$ and for $i\in\{0,N_y-1\}$ set $t_i:=\max\{t\in [0,1]:\delta_\Omega(\gamma(t))=D e^{-i}\}$, $y_i:=\gamma(t_i)$. Furthermore, let $t_{-N_x}:=0$, $t_{N_y}:=1$,$x_{-N_x}:=x$ and $y_{N_y}=y$.
We set
$\gamma_i:=\gamma|_{[t_i,t_{i+1}]}$ if $i\leq -1$, $\gamma_i:=\gamma|_{[t_{i-1},t_{i}]}$ if $i\geq 1$ and $\gamma_0$ to be the remaining piece of $\gamma$, i.e. the piece of $\gamma$ that connects $x_0$ to $y_0$. Also let $D_i:=D e ^{i}$ if $i\leq 0$ and $D_i= D e^{-i}$ if $i\geq 1$. By construction we have
\begin{equation}\label{dividingcurve}
	\delta_\Omega(z)\leq e D_i \:\:\: \text{if} \:\:\: z\in\gamma_i \:\:\:\:\: \text{and} \:\:\:\:\:
\delta_{\Omega}(z)\in \{ D_i, e D_i\} \:\:\: \text{if $z$ is an endpoint of} \: \gamma_i.
\end{equation}

Our proof of Theorem \ref{firstbig} is heavily inspired from \cite{LWZ}, where they follow the ideas given in \cite{GH}. Essentially, the lemma below allows us to get somewhat better bounds.

\begin{lemma}\label{stronggoldi4} There exists $C>0,$ \textit{depending only on $\O$ and $\lambda$,} such that $\forall i\in\{-N_x,...,N_y\}$ we have
	$$
	\el(\gamma_i)\leq C \w(D_i)\log\left(\dfrac{\w(D_i)}{D_i}\right).
	$$
\end{lemma}
\begin{proof}
	As $\gamma$ is a $\lambda$-geodesic it follows that each $\gamma_i$ is also a $\lambda$-geodesic. Thus, by Lemma \ref{dinilemma}, \eqref{dividingcurve} implies that there exists $C_0>0$ independent from $\gamma$ such that \begin{equation}\label{stronggoldi1} \krl(\gamma_i)\leq \lambda \log\left(1+\dfrac{C_0 \el (\gamma_i)}{D_i}\right).
	\end{equation}
	
	On the other hand, by Definition \ref{stronggoldi} and \eqref{dividingcurve} we have
	\begin{equation}\label{stronggoldi2} \el(\gamma_i) \leq \w(e D_i) \krl(\gamma_i).
	\end{equation} Setting $C_2:=\max\{C_0,\lambda e\}$, by \eqref{wecangettheconstantout}, \eqref{stronggoldi1} and \eqref{stronggoldi2} we get
	\begin{equation}\label{stronggoldi3}
		\el(\gamma_i)\leq C_2 \w(D_i) \log\left(1+\dfrac{C_2 \el (\gamma_i)}{D_i}\right).\end{equation}
		
	\textit{Claim.} For any $C>0$, there exists $C'>0$ such that if $x \geq 0$, $y \geq e$ and $x \leq C y \log(1+x)$, then $x \leq C' y \log(y)$.
	
	To see this, set $f(x):=x/(C\log(1+x))$ and $g(x):=\max\{e,f(x)\log(f(x))\}$. Observe that $h(x):=x/g(x)$ is a positive continuous function on $\mathbb{R}^+_0$ and $\lim_{x\to`\infty} h(x)= C$. Then $C':= \sup h(x) < \infty$. Since $t \log t$ is an increasing function for $t\geq e^{-1}$ and $y\geq \max\{e,f(x)\}$ we obtain $x \leq C' g(x) \leq C' y \log(y)$.
	
	Recall that we assume that $\w(r)/r \geq e$. Then the lemma follows from the claim and \eqref{stronggoldi3} by setting $C:=C^2_2$, $x:=C_2 \el(\gamma_i)/D_i$ and $y:=\w(D_i)/D_i$. \end{proof}
	
\begin{lemma}\label{bound} Let $z\in\gamma_i$. Then there exists $C > 0$, depending only on $\O$ and $\lambda$, such that $$ \min \{\el(\gamma_{x,z}) , \el(\gamma_{z,y})\} \leq C \g(D_i),$$ where $\gamma_{x,z}$ and $\gamma_{z,y}$ denote the pieces of $\gamma$ joining $x$ to $z$ and $z$ to $y$ respectively.
\end{lemma}
\begin{proof}
	Without loss of generality we assume that $i\leq 0$. Then by Lemma \ref{stronggoldi4} there exists $C>0$, depending only on $\O$ and $\lambda$, such that for each $j\leq i$ we have $$\el(\gamma_j)\leq C \w(D_j)\log\left(\dfrac{\w(D_j)}{D_j}\right).$$
	
	Thus, due to Definition \ref{stronggoldi} we have $$\min\{\el(\gamma_{x,z}),\el ( \gamma_{z,y})\}\leq\el(\gamma_{x,z}) \leq \displaystyle{\sum_{j=-N_x}^{i}\el(\gamma_{j})}\leq$$ $$  \displaystyle{ C\sum_{j=-N_x}^{i}\w(D e ^{j})\log\left(\dfrac{\w(De^{j})}{De^{j}}\right) } \leq	
	\displaystyle{ C\sum_{j=-\infty}^{i}\w(De^{j})\log\left(\dfrac{\w(De^{j})}{De^{j}}\right)} \leq$$
	$$ \displaystyle{ C\int_{-\infty}^{i}\w(De^{t})\log\left(\dfrac{\w(De^{t})}{De^{t}}\right)}dt\displaystyle{\leq C\int_{0}^{D_i}\left(\dfrac{\w(r)}{r}\right)\log\left(\dfrac{\w(r)}{r}\right)dr}\leq C  \g(D_i).$$
\end{proof}

\textit{Proof of Theorem \ref{firstbig}.}
Let $z\in \gamma_i$ be a point in the image of $\gamma$ with $\el(\gamma_{x,z})=\el(\gamma_{z,y}),$ where $\gamma_{x,z}$ and $\gamma_{z,y}$ denote the pieces of $\gamma$ joining $x$ to $z$ and $z$ to $y$ respectively. By Proposition \ref{propertiesofstronggoldi} and Lemma \ref{bound} it follows that
$$\el(\gamma)=\el(\gamma_{x,z})+\el(\gamma_{z,y})\leq 2 C \g(D_i) \leq 2 C \g (D_\gamma). $$

$\: \endofproof $

We have the following immediate corollary of Theorem \ref{firstbig} and Proposition \ref{cconvex}.

\begin{cor}\label{curvegoesinside}
		Let $\Omega$ be a SG domain and $\gamma:[0,1]\rightarrow\Omega$ be a $\lambda$-geodesic joining $x,y\in\Omega$. Then there exists
$c_0>0,$ depending only on $\O$ and $\lambda,$ such that
		$$ D_\gamma \ge c_0\g^{-1}(\|x-y\|) .$$

In particular, if $\Omega$ is a $SG_\epsilon$ domain and $c\in(0,1),$ then
$$D_\gamma \gtrsim\left(\frac{\|x-y\|}{\log(1/\|x-y\|)}\right)^{1/\epsilon},\quad \|x-y\|\le c.$$
\end{cor}

\section{A Gehring-Hayman type theorem}\label{ghaymanthms}
Recall the classical Gehring-Hayman inequality for simply connected planar domains.
\begin{thm}\cite{GH}\label{thm:ofgehringhayman}
	Let $G\subsetneq\C$ be a simply connected planar domain. There exists $C>0$ not depending on $G$ such that the inequality
	$$\el(\gamma)\leq C\el(\gamma')$$ holds, where $\gamma:[0,1]\to G$ is a Kobayashi geodesic and $\gamma':[0,1]\to G$ is any other curve joining the same two points.
\end{thm}

For a deeper discussion on the above result see the survey \cite{PR}.

In \cite{LWZ}, the authors provided extensions of Theorem \ref{thm:ofgehringhayman} to strongly pseudoconvex domains and bounded $m$-convex domains with Dini-smooth boundary in $\Cn$. In this section we will extend and improve \cite[Theorem 4.1]{LWZ} and \cite[Theorem 1.5]{LWZ}.

Let $\O$ be a domain in $\Cn$ and recall that $p\in\partial\O$ is a convexifiable boundary point if there a exist a neighbourhood of $p$, $U_p$ and a biholomorphism $\Phi:U_p\rightarrow D$ such that $\Phi(p)$ is a convex boundary point of $\Phi(\Omega\cap U_p)$. We say that $\Omega$ is convexifiable if any $p\in\partial \Omega$ is a convexifiable boundary point. The goal of this section is to prove the following result, which extends Theorem \ref{thm:ofgehringhayman} to convexifiable SG domains.

 \begin{thm}\label{secondbigcfb}
 	Let $\Omega$ be a convexifiable SG domain and $\gamma:[0,1]\rightarrow\Omega$ be a Kobayashi geodesic joining $x, y\in\Omega$. Then there exists $C > 0$, depending only on $\O$, such that $$ \el(\gamma)\leq C \g(\|x-y\|).$$ Furthermore, if $\Omega$ is Gromov hyperbolic, then the same assertion holds for $\lambda$-geodesics (with different $C$ that also depends on $\lambda$).
 \end{thm}

Combining Proposition \ref{cconvex} and Theorem \ref{secondbigcfb} gives:
\begin{cor}\label{betterboundsthanchineseguys} Let $c\in(0,1).$
\smallskip
	
	(i) Let $\Omega$ be a strongly pseudoconvex domain and $\gamma$ be a $\lambda$-geodesic joining $x$ to $y$. Then there exists $C>0,$ depending only on $\O,$ $\lambda,$ and $c,$ such that
			$$\el(\gamma)\le C\|x-y\|^{\frac{1}{2}}\log(1/\|x-y\|),\quad \|x-y\|\le c.$$
	
	(ii) Let $\Omega$ be an $m$-convex domain with Dini-smooth boundary and $\gamma$ be a Kobayashi geodesic joining $x$ to $y$. Then there exists $C>0,$ depending only on $\O$ and $c,$ such that
		$$\el(\gamma)\le C\|x-y\|^{\frac{1}{m}}\log(1/\|x-y\|), \quad \|x-y\|\le c.$$
		If $\Omega$ is also Gromov hyperbolic, the same assertion holds for $\lambda$-geodesics (with different constants that also depend on $\lambda$).
\end{cor}

Corollary \ref{betterboundsthanchineseguys} gives much better estimates than \cite[Theorem 1.5]{LWZ} and \cite[Theorem 4.1]{LWZ}.

We would like to note that after this paper appeared, \cite[Theorem 1.7]{LiuPW} improved the estimates given in Corollary \ref{betterboundsthanchineseguys} (with a different approach than ours) removing the logarithmic terms. Same estimates remain true for smooth bounded pseudoconvex domains of finite type in $\C^2$, see \cite[Theorem 1.2]{LiPW}.

The proof of Theorem \ref{secondbigcfb} is motivated by proof of this statement on the simpler yet less general case where $\Omega$ is convex. For convenience, we have two subsections. In the next subsection, we will first prove Theorem \ref{secondbigcfb} with the additional assumption that $\Omega$ is convex, then we will extend our proof to the general case.

\subsection{Proof of Theorem \ref{secondbigcfb} in the convex case} Note that since $\O$ is a convexifiable bounded domain it is complete hyperbolic, hence any $x,y\in\O$ is joined by a Kobayashi geodesic. Moreover, by Lempert's Theorem (see e.g. \cite[Theorem 11.2.1]{JP}) $\kappa_\Omega=\hat\kappa_\Omega$, then Remark \ref{rem:1geodesics} implies that the Kobayashi geodesics of $\O$ are also $1$-geodesics.

Having this in mind, let $\gamma:[0,1]\rightarrow\O$ be a Kobayashi geodesic joining $x$ to $y$ and set $D:=D_\gamma$. Choose $z\in\Omega\cap\gamma(I)$ such that $\delta_{\Omega}(z)=D$. By Lemma \ref{convexlemma} and Lemma \ref{dinilemma} we have
	
	$$
		\log\left(\dfrac{D}{\sqrt{\delta_{\Omega}(x)\delta_{\Omega}(y)}}\right)= \dfrac{1}{2}\log\left(\dfrac{D}{\delta_{\Omega}(x)}\right)+\dfrac{1}{2}\log\left(\dfrac{D}{\delta_{\Omega}(y)}\right)
	 $$ $$\leq k_\Omega(x,z)+k_\Omega(z,y) = k_\Omega(x,y)\leq  \log\left(1+\dfrac{C\|x-y\|}{\sqrt{\delta_{\Omega}(x)\delta_{\Omega}(y)}}\right).
	$$
	
	So we obtain $$ D\leq C\|x-y\|+\sqrt{\delta_{\Omega}(x)\delta_{\Omega}(y)} .$$

	\textit{Case I.} First, we assume that the points are close, that is,
	 \begin{equation}\label{thm2bound2}\sqrt{\delta_{\Omega}(x)\delta_{\Omega}(y)} \geq \|x-y\| \:\:\:\:\: \text{so} \:\:\:\:\:
		D\leq (C+1)\sqrt{\delta_{\Omega}(x)\delta_{\Omega}(y)}.
	\end{equation}

	By Lemma \ref{dinilemma}, Definition \ref{stronggoldi} and the fact that $\log(1+x)\leq x$ for $x\geq 0$ we have that
	\begin{equation}\label{boundinglengthofgeodesics}
		\dfrac{\el(\gamma)}{\w(D)}\leq\displaystyle{\int_{0}^1 \kappa_\Omega(\gamma(t);\gamma'(t))dx=  k_\Omega(x,y)} \leq \frac{C\|x-y\|}{\sdd}. \end{equation}
	
    Recall that by Remark \ref{rem:remarkaboutSGdomains} the function $\w$ is increasing. Then by \eqref{thm2bound2} and \eqref{boundinglengthofgeodesics} we get
	$$\el(\gamma) \leq C\frac{\w(D)\|x-y\|}{\sdd)}\leq C\frac{\w((C+1)\sdd)\|x-y\|}{\sdd}.$$
	
	Continuing above and using \eqref{wecangettheconstantout}, by the assumption that $\sdd \geq \|x-y\|$ we obtain $$\el(\gamma)\leq C\frac{\w((C+1)\sdd)\|x-y\|}{\sdd}$$ $$\leq  C(C+1) \frac{\w(\sdd)\|x-y\|}{\sdd}\leq  C(C+1) \w(\|x-y\|) .$$
	
	By Proposition \ref{propertiesofstronggoldi} this case follows.
	
	\textit{Case II.} Now, we assume that the points are far, that is, \begin{equation}\label{thm2bound3}
		\sqrt{\delta_{\Omega}(x)\delta_{\Omega}(y)}\leq \|x-y\| \:\:\:\:\: \text{so} \:\:\:\:\: D\leq (1+C) \|x-y\|.
	\end{equation}

Using \eqref{wecangettheconstantout} again, by Theorem \ref{firstbig} and \eqref{thm2bound3}, we immediately have $$ \el(\gamma)\leq C' \g(D) \leq C' \g((1+C)\|x-y\|) \leq C' (1+C) \g(\|x-y\|). $$

	As $\gamma$ is chosen arbitrarily, the theorem follows for Kobayashi geodesics.
	
    Now we move on to the case of $\lambda$-geodesics. Assume that $\Omega$ is Gromov hyperbolic and let $\sigma:[0,1]\rightarrow\O$ be a $\lambda$-geodesic joining $x'$ and $y'$. Set $D':=D_\sigma$ and let $z'$ be a point in the image of $\sigma$ satisfying $\delta_{\Omega}(z')=D'$. Consider a  Kobayashi geodesic $\gamma':[0,1]\to\O$ joining $x'$ to $y'$. Lemma \ref{geodesicstabilitylemma} implies that there exists $w$ in the image of $\gamma'$ with $k_\Omega(z',w)\leq K$, where $K>0$ is a constant depending only on $\O$ and $\lambda$. Then by Lemma \ref{convexlemma} $$\dfrac{1}{2}\log\left(\dfrac{\delta_{\Omega}(z')}{\delta_{\Omega}(w)}\right)\leq k_\Omega(z',w) \leq K < \infty.$$
Hence $ D'=\delta_{\Omega}(z')\leq e^{2K} \delta_{\Omega}(w) \leq e^{2K} D_{\gamma'}$.

    Therefore by \eqref{wecangettheconstantout} and above we have \begin{equation}\label{gromovconsequence} \w(D')\leq e^{2K}\w(D_{\gamma'}) \: \text{and} \: \g(D')\leq e^{2K} \g(D_{\gamma'}).
    \end{equation}

    By \eqref{gromovconsequence}, the estimate $$\el(\sigma)\leq  \lambda \w(D')\log\left(1+\dfrac{C\|x'-y'\|}{\sqrt{\delta_\O(x')\delta_\O(y')}}\right),$$ and the estimate $\el(\sigma)\leq C \g(D')$ (which is the conclusion of Theorem \ref{firstbig}) the theorem follows by repeating arguments from above. $\endofproof$

    We now move on to the next subsection, where we establish an analogue of Lemma \ref{convexlemma} for convexifiable SG domains, then prove Theorem \ref{secondbigcfb} in the general case.
\subsection{Proof of Theorem \ref{secondbigcfb}}
The following notion introduced in \cite{BNT} will be useful. Let us recall that $p\in\partial\Omega$ is said to be a $k$-point if for any neighbourhood $U_p$ of $p$, we have that $$ \displaystyle{ \liminf_{x\rightarrow p} \left(K_{\Omega}(x,\Omega\setminus U_p)-\dfrac{1}{2}\log\left(\dfrac{1}{\delta_\Omega(x)}\right)\right) > -\infty}.$$

The following lemma follows from definition.
\begin{lemma}\label{klemma}
	If $p\in\partial \Omega$ is a k-point, then for any neighbourhood $U_p$ of $p$ there exist another neighbourhood $U'_p\subset\subset U_p$ of $p$ and a constant $c > 0$ such that $$k_\Omega(x,y)\geq \dfrac{1}{2}\log\left(\dfrac{1}{\delta_\Omega(x)}\right)-c, \:\:\: x\in \Omega \cap U'_p, \: y\in\Omega\setminus U_p. $$
\end{lemma}

We have the following localization result. We note that the result given as reference is stronger however for our purposes the version given below is sufficient.

\begin{lemma}\cite[Theorem 6.15]{BNT}\label{lem:localization}
	Let $\Omega$ be a domain in $\Cn$ with Dini-smooth boundary and assume that $p\in\partial\Omega$ is a $k$-point. Then for any neighbourhood $U_p$ of $p$, there exist another neighbourhood $V_p\subset\subset U_p$ and a constant $c > 0$ such that we have \begin{equation}\label{localization}
		k_\Omega(x,y)\geq k_{\Omega\cap U_p}(x,y)-c, \:\:\:\:\: x,y \in\Omega\cap V_p.
	\end{equation}
\end{lemma}

The following lemma follows from the more general result given in \cite[Theorem 6.7]{BNT}.

\begin{lemma}\cite[Theorem 6.7]{BNT}\label{localkglobalk}
	Let $\Omega$ be a bounded domain, $p\in\partial \Omega$ and $U_p$ be a neighbourhood of $p$ such that $p$ is a $k$-point for $\Omega\cap U_p$. If $\Omega\cap U_p$ has $\alpha$-growth, then $p$ is also a $k$-point for $\Omega$.
\end{lemma}
\begin{lemma}\label{anypointiskpoint}
	Let $\Omega$ be a convexifiable SG domain. Then any $p\in\partial \Omega$ is a $k$-point.
\end{lemma}
\begin{proof}
	A key observation which will be used in the proof of this lemma is that if $\Phi$ is a bi-Lipschitz map of closures of two domains $D_1,D_2\subset \mathbb{R}^n$ w.r.t. the Euclidean distance, then
	\begin{equation}\label{important} \delta_{D_1}(x)\sim\delta_{D_2}(\Phi(x))\:\:\:\:\:x\in D_1.
	\end{equation}

	Let $\Omega$ be as above and $p\in\partial\Omega$ be arbitrary. We start the proof by considering a small enough neighbourhood of $p$, $U_p$ so that
	$$\Psi: U_p \rightarrow \Psi(U_p)$$ is biholomorphic and $D:=\Psi(\Omega\cap U_p)$ is a convex domain. As $\Omega$ is a SG domain, by Proposition \ref{propertiesofstronggoldi}, for $x\in \Omega\cap U_p$ close to $p$ we have
	\begin{equation}\label{growthforlemma}
		\kappa_{\Omega\cap U_p}(x;v) \geq \kappa_\Omega(x;v) \geq \dfrac{\|v\|}{\w(\delta_\Omega(x))} = \dfrac{\|v\|}{\w(\delta_{\Omega\cap U_p}(x))}. \end{equation}
	
Take an Euclidean ball $B$ centered at $\Psi(p)$. Consider the domain  $G:=D\cap B$. By convexity, it has $\alpha$-growth for some $\alpha > 0$. Moreover, by \eqref{growthforlemma} and biholomorphic invariance of the Kobayashi-Royden pseudometric, following the argument in the proof of \cite[Corollary 2.4]{BZ} we see that $G$ also satisfies Condition 1 in Definition \ref{goldilocksdefn}. Then $G$ is Goldilocks and, in particular, it satisfies the visibility property. As $G$ is also convex, by  \cite[Proposition 6.2]{BNT} it follows that any boundary point of $G$ is a $k$-point.
	
	By the discussion above, it follows by \eqref{important} that any boundary point of $\Psi^{-1}(G)$ is a $k$-point and $\Psi^{-1}(G)$ has $\alpha$-growth. Applying Lemma \ref{localkglobalk} with $\Omega$ and $\Psi^{-1}(G)$ we see that $p$ is also a $k$-point for $\Omega$. As $p\in\partial\Omega$ was arbitrary, the lemma follows.
\end{proof}
The following Theorem is the analogue of Lemma \ref{convexlemma} for convexifiable SG domains, in the subcase of strongly pseudoconvex domains it follows from the stronger estimate given in \cite{BB} with constant $c=0$.

\begin{thm}\label{convexgoldilocks}
	Let $\Omega$ be a convexifiable SG domain. Then there exists $c>0$ such that  $$k_\Omega(x,y)\geq\dfrac{1}{2}\log\left(\dfrac{\delta_\Omega(y)}{\delta_\Omega(x)}\right)-c,\:\:\:\:\: x,y\in\O.$$
\end{thm}
\begin{proof}
	We suppose the contrary. Then we can find two sequences $x_n,y_n\in\O$ converging to $p,q\in\overline{\Omega}$ respectively which satisfies $$k_\Omega(x_n,y_n)<\dfrac{1}{2}\log\left(\dfrac{\delta_\Omega(y_n)}{\delta_\Omega(x_n)}\right)-n.$$

	By Lemma \ref{klemma} and Lemma \ref{anypointiskpoint} , it follows that if $x_n,y_n$ tend to $p\neq q \in \overline{\Omega}$ we can find a constant $c > 0$ depending on $p,q$ such that $$k_\Omega(x_n,y_n)\geq \dfrac{1}{2}\log\left(\dfrac{1}{\delta_\Omega(x_n)}\right)+ \frac{1}{2}\log\left(\dfrac{1}{\delta_\Omega(y_n)}\right)-c $$ so we can only have the case where $p=q \in \partial \Omega$.
	
	Choose a small enough neighbourhood $U_p$ so that there is a biholomorphism of $U_p$ mapping $\O\cap U_p$ onto a convex domain.
	By Lemma \ref{anypointiskpoint}, taking $U_p$ small enough we can find another neighbourhood $V_p\subset\subset U_p$ of $p$ such that \eqref{localization} holds for the pair $\O$ and $\O\cap U_p$ on $\O\cap V_p$. Then the proof of Lemma \ref{anypointiskpoint} implies that if $ x_n,y_n \in \O\cap V_p$ we have
	$$k_\Omega(x_n,y_n)\geq k_{\Omega\cap U_p}(x_n,y_n)-c $$ $$\geq \dfrac{1}{2}\log\left(\dfrac{\delta_{\Omega\cap U_p(y_n)}}{\delta_{\Omega\cap U_p(x_n)}}\right) - c' - c = \dfrac{1}{2}\log\left(\dfrac{\delta_\Omega(y_n)}{\delta_\Omega(x_n)}\right) - c' - c $$
	so we see that $p=q\in\partial\Omega$ case is also impossible.
\end{proof}	
	
Now, we can reach the desired conclusion. \\

\textit{Proof of Theorem \ref{secondbigcfb}.}
Theorem follows by imitating the proof of the convex case, by applying Theorem \ref{convexgoldilocks} instead of Lemma \ref{convexlemma}. $ \endofproof$

\section{Lower Bounds of the Kobayashi distance}\label{kobayashilowerbounds}
In this subsection we will provide two different lower bounds for the Kobayashi distance in a convexifiable SG domain $\O$, in terms of the functions $\omega_\O$ and $g_\O$. We will later merge them to get a unified lower bound.

Theorem \ref{firstbig} combined with Lemma \ref{convexlemma} gives:
\begin{prop}\label{lowerboundgood1}
	Let $\Omega$ be a convexifiable SG domain. Then there exists $c > 0$ such that $$ k_\Omega(x,y)\geq   \dfrac{1}{2}\log\left(\dfrac{\g^{-1}(c\|x-y\|)}{\delta_\Omega(x)}\right), \:\:\:\:\: x,y\in\O.$$
\end{prop}
\begin{proof}
	Let $\gamma$ be a Kobayashi geodesic joining $x$ and $y$ and let $z$ be the point on image of $\gamma$ satisfying $\delta_{\Omega}(z)=D_\gamma$.
	
By Theorem \ref{firstbig}, we have that $\|x-y\| \leq \el(\gamma) \leq C\g(D_\gamma).$ Then,
by Theorem \ref{convexgoldilocks},  $$ k_\Omega(x,y) \geq k_\Omega (x,z) \geq\dfrac{1}{2}\log\left(\dfrac{c'D_\gamma}{\delta_{\Omega}(x)}\right) \geq  \dfrac{1}{2}\log\left(\dfrac{\g^{-1}(c\|x-y\|)}{\delta_{\Omega}(x)}\right),$$
where $c=c'/C.$ (Here we assume that $c'\le 1;$ then $c'\g^{-1}(r)\ge\g^{-1}(c'r)$
as a consequence of \eqref{wecangettheconstantout}).
\end{proof}

\begin{prop}\label{lowerboundugly1}
	Let $\Omega$ be a convexifiable SG domain. Then there exists $c>0$ such that for all $x,y\in \Omega$ we have that
	$$ k_{\Omega}(x,y)\geq \dfrac{1}{2}\log\left(1+\dfrac{c\|x-y\|}{\w(\delta_\Omega(x))}\right),\:\:\:\:\: x,y\in\O .$$
\end{prop}
\begin{proof}
	For convenience, let us first start with the simple case where $\Omega$ is convex. By definition we have \begin{equation}\label{effective}\dfrac{1}{\delta_{\Omega}(x;v)}\geq \kappa_\Omega(x;v)\geq \dfrac{\|v\|}{\w(\delta_{\Omega}(x))}
	\end{equation}
	so Lemma \ref{lem:proofisgiven} and Proposition \ref{propertiesofstronggoldi} directly implies that $$k_\Omega(x,y) \geq \dfrac{1}{2}\log\left(1+\dfrac{1}{\delta_\Omega\left(x;x-y\right)}\right) \geq \dfrac{1}{2}\log\left(1+\dfrac{\|x-y\|}{\w(\delta_\Omega(x))}\right).$$
	
	So in this case proposition follows.
	
	We will prove the convexifiable case by contradiction. Suppose that such an estimate does not hold, that is, there exists a sequence of points $x_n,y_n\in\O$ tending to $p,q\in\overline{\Omega}$ respectively and a sequence $C_n\rightarrow 0$ such that $$ k_\Omega(x_n,y_n)\leq \dfrac{1}{2}\log\left(1+\dfrac{C_n\|x_n-y_n\|}{\w(\delta_\Omega(x_n))}\right).$$
	
	It follows from the inequality $k_\Omega(x,y)\geq c\|x-y\|$ (see e.g. \cite[Proposition 3.5]{BZ}) valid on any bounded domain that we must have $p\in\partial \Omega$.
	
	Furthermore, by Lemma \ref{anypointiskpoint} we see that $p$ is a $k$-point, so if $x_n,y_n\rightarrow p,q$ and $p\neq q$ we have $$k_\Omega(x_n,y_n)\geq\dfrac{1}{2}\log\left(\dfrac{1}{\delta_\Omega(x_n)}\right)-c\geq\dfrac{1}{2}\log\left(\dfrac{1}{\w(\delta_\Omega(x_n))}\right)-c$$ therefore we must have $q=p\in\partial\Omega$.
	
	From now on our approach will be similar to the proof of Theorem \ref{convexgoldilocks}.
	Note that if $f:D_1\rightarrow D_2$ is a biholomorphism of the domains $D_1,D_2\subset\Cn$ that extends to the boundary, then
	\begin{equation}\label{important2}\|x-y\|\sim\|f(x)-f(y)\|\:\:\:\:\:\text{and}\:\:\:\:\:
		\dfrac{\|x-y\|}{	\delta_{D_1}\left(x;\dfrac{v}{\|v\|}\right)}\sim\dfrac{\|f(x)-f(y)\|}{	\delta_{D_2}\left(f(x);\dfrac{df_x v}{\|df_x v \|}\right)}, \:\:\:\:\: x,y\in D_1, \:\:\:v\in\Cn.\end{equation}
	
	Now, we take a neighbourhood of $p$, $U_p$ so that $\Psi$ is a biholomorphism of $U_p$ so that $G:=\Psi(\O\cap U_p)$ is convex. By \eqref{important2}, applying Lemma \ref{lem:proofisgiven} to $G$ and Lemma \ref{lem:localization} to $\Omega$ and $\Omega\cap U_p$ gives $$ k_\Omega(x,y) \geq
	k_{\Omega\cap U_p}(x,y) - c = k_G(\Psi(x),\Psi(y)) - c $$ \begin{equation}\label{hit}\geq
		\dfrac{1}{2}\log\left(1+\dfrac{c' \|\Psi(x)-\Psi(y)\|}{\delta_{\Omega}\left(x;d\Psi^{-1}_{\Psi(x)}\dfrac{\Psi(x)-\Psi(y)}{\|\Psi(x)-\Psi(y)\|}\right)}\right) - c.
	\end{equation}

	Due to \eqref{effective}, \eqref{hit} shows that $p=q$ case is also impossible. This concludes the proof of the proposition.
\end{proof}
Although they are proven with global assumptions, we remark that the estimates above are also of local character. That is to say, if $U_p$ is a small enough convexifiable neighbourhood of $p\in\partial \Omega$ is and if the SG domain conditions are satisfied near $p$, there is another neighbourhood of $p$, $U'_p\subset \subset U_p$ such that for $x,y\in U'_p$ the estimates given above are satisfied. To see this one may use the localizations $k_\Omega\geq k_{\Omega\cap U_p}-c$ given in Lemma \ref{lem:localization} and $ck_\Omega\geq k_{\Omega\cap U_p}$ due to \cite{FR}. For details, the interested reader may examine the proofs of \cite[Theorem 1.4]{BNT} and \cite[Theorem 1.7]{NT}.

Another thing to note is that these estimates are only good when $x,y$ are close. By Lemma \ref{anypointiskpoint}, any $p\in\partial \Omega$ is a $k$-point, thus by Lemma \ref{klemma} when $x$ and $y$ are far we have the much better estimate $$k_\Omega(x,y)\geq \dfrac{1}{2}\log\left(\dfrac{1}{\delta_\Omega(x)}\right)+\dfrac{1}{2}\log\left(\dfrac{1}{\delta_\Omega(y)}\right)-c.$$

It is not difficult to see that the proof of Proposition \ref{lowerboundugly1} implicitly suggests that an analogue of Lemma \ref{lem:proofisgiven} holds on convexifiable SG domains. This means, in some sense, that Proposition \ref{lowerboundugly1} is not sharp. We include it in order to compare it with Proposition \ref{lowerboundgood1} to get a unified estimate. To state our results, we set
$$h_{\Omega,1}(x,y,c):=\dfrac{c\|x-y\|}{\g(\delta_\Omega(x))}, \: \: h_{\Omega,2}(x,y,c):=\dfrac{ \g^{-1}(c\|x-y\|)}{\delta_\Omega(x)},$$
and $h_\Omega:=\max\{h_{\Omega,1},h_{\Omega,2}\}$. With this notation we have:
\begin{thm}\label{lowerboundfinal}
	Let $\Omega$ be a convexifiable SG domain. There exists $c>0$ such that the Kobayashi distance on $\Omega$ satisfies
	$$k_\Omega(x,y) \geq \dfrac{1}{2}\log\left(1+h_\Omega(x,y,c)\right)\left(1+h_\Omega(y,x,c)\right),\:\:\:\:\: x,y\in\O.$$
\end{thm}

\begin{proof}
We will prove this corollary by comparing the lower bounds given in Proposition \ref{lowerboundgood1} and Proposition \ref{lowerboundugly1}.

It follows by the definition of $k_\Omega$ that for any $x,y\in \Omega$ and any $\epsilon>0$
one may find a point $z\in D$ such that $\|x-z\|=\|y-z\|$ and $k_\Omega(x,y)>k_\Omega(x,z)+k_\Omega(y,z)-\epsilon.$

Set $c:=\frac{c'}{4},$ where $c'>0$ is chosen such that Proposition \ref{lowerboundgood1} and Proposition \ref{lowerboundugly1}
hold. As the functions $\g^{-1}(r)/r$ and $r/\g(r)$ are increasing by Proposition \ref {propertiesofstronggoldi},
we see that we either have
$$h_{\Omega,2}(x,y,c)\geq h_{\Omega,1}(x,y,c) \geq 1  \: \text{or} \: h_{\Omega,2}(x,y,c) \leq h_{\Omega,1}(x,y,c) \leq 1.$$

\textit{Case I.} $h_{\Omega,2}\geq h_{\Omega,1}\geq 1$.

Proposition \ref{lowerboundgood1} leads to
$$k_\Omega(x,z)\geq\dfrac{1}{2}\log(h_{\Omega,2}(x,z,c'))\geq\dfrac{1}{2}\log(h_{\Omega,2}(x,y,2c))\geq$$
$$\dfrac{1}{2}\log(2h_{\Omega,2}(x,y,c))\geq \dfrac{1}{2}\log(1+h_{\Omega,2}(x,y,c)).$$

\textit{Case II.} $1 \geq h_{\O,1} \geq h_{\O,2}$.

By Proposition \ref {propertiesofstronggoldi} we have $\g\geq\w$, so
Proposition \ref{lowerboundugly1} immediately implies that
$$k_\Omega(x,z)\geq\dfrac{1}{2}\log\left(1+\dfrac{c'\|x-z\|}{\w(\delta_\Omega(x))}\right)\ge
\dfrac{1}{2}\log(1+h_{\Omega,1}(x,y,c)).$$

In both cases we obtain $$k_\Omega(x,z) \geq
	\dfrac{1}{2}\log\left(1+h_\Omega(x,y,c)\right). $$

Analogously, $$k_\Omega(y,z) \geq 	\dfrac{1}{2}\log\left(1+h_\Omega(y,x,c)\right).$$

Thus
$$k_\Omega(x,y)\geq \dfrac{1}{2}\log\left(1+h_\Omega(x,y,c)\right)\left(1+h_\Omega(y,x,c)\right)-\epsilon.$$

Letting $\epsilon\to 0$ the theorem follows.
\end{proof}

In order to make a final remark we recall the following result:
\begin{thm}\label{pisapaper}\cite[Theorem 1.6]{NT}
	Let $\Omega$ be a \emph{($\mathbb{C}$-)}convexifiable domain of finite $m$-type. Then there exists $c>0$ such that the Kobayashi distance on $\Omega$ satisfies
	$$ k_\Omega(x,y)\geq \dfrac{m}{2}\log\left(1+ \dfrac{c\|x-y\|}{\delta^{\frac{1}{m}}_{\Omega}(x)}\right)\left(1+ \dfrac{c\|x-y\|}{\delta^{\frac{1}{m}}_{\Omega}(y)}\right),\:\:\:\:\: x,y\in\O. $$
\end{thm}
 If $\Omega$ is $m$-convex with Dini-smooth boundary we have $$h_\Omega(x,y,c)=\max\left\{\dfrac{(c\|x-y\|) ^m}{(-\log(c\|x-y\|))^m\delta_\Omega(x)},\dfrac{c\|x-y\|}{-\log(\delta_\Omega(x))\delta^{\frac{1}{m}}_\Omega(x)}\right\} .$$

Thus, we see that ignoring the "parasite" logarithmic terms, on Dini-smooth $m$-convex case, Theorem \ref{lowerboundfinal} gives the same estimate as Theorem \ref{pisapaper}. In particular, by the examples given in \cite[Section 1.7]{NT} our estimates are pretty sharp. Although our result is slightly worse, it holds for domains with much lower boundary regularity. Moreover, our construction avoids the machinery of peak functions.
\smallskip
\subsection*{Acknowledgements}
We would like to thank to Pascal J. Thomas for many helpful conversations which lead to improvement of exposition of this manuscript. We are also grateful to the referee for a careful reading of the manuscript and for
detailed comments and suggestions which helped us to improve the manuscript.

{}

\end{document}